\documentclass[reqno]{amsart}

\usepackage{latexsym}
\usepackage{amsmath}
\usepackage{amssymb, amscd, amsthm}
\usepackage[all]{xy}

\newtheorem{theo}{{Theorem}}[section]
\newtheorem{cor}[theo]{{Corollary}}
\newtheorem{lemma}[theo]{{Lemma}}
\newtheorem{pro}[theo]{Proposition}

\theoremstyle{remark}
\newtheorem{defn}[theo]{Definition}

\theoremstyle{remark}
\newtheorem*{remark}{Remark}

\newcommand{\ra}{\rightarrow}

\newcommand{\Lra}{\Longleftrightarrow}
\newcommand{\sto}[1][]{\stackrel{#1}{\to}}

\newcommand{\cM}{\mathcal{M}}

\newcommand{\F}{\mathbb F}
\newcommand{\N}{\mathbb N}
\newcommand{\Q}{\mathbb Q}
\newcommand{\R}{\mathbb R}
\newcommand{\Z}{\mathbb Z}

\DeclareMathOperator{\diag}{\mathrm{diag}}
\DeclareMathOperator{\ord}{\mathrm ord}
\DeclareMathOperator{\rank}{\mathrm rank}
\DeclareMathOperator{\tor}{\mathrm tor}

\DeclareFontEncoding{OT2}{}{} 
  \newcommand{\textcyr}[1]{%
    {\fontencoding{OT2}\fontfamily{wncyr}\fontseries{m}\fontshape{n}%
     \selectfont #1}}
\newcommand{\Sha}{{\mbox{\textcyr{Sh}}}}

\begin{document}
\title[On non-congruent numbers]{On non-congruent numbers \\ with $1$ modulo $4$ prime factors}
\author{Yi Ouyang and  Shenxing Zhang}
\address{Wu Wen-Tsun Key Laboratory of Mathematics,  School of Mathematical Sciences, University of Science and Technology of China, Hefei, Anhui 230026, China}

\email{yiouyang@ustc.edu.cn}
\email{zsxqq@mail.ustc.edu.cn}
\thanks{Research partially
supported by Project 11171317 from NSFC}
\date{\today}
\subjclass[2010]{Primary 11G05; Secondary 11D25}
\maketitle

\begin{abstract}
In this paper, we use the $2$-decent method to find a series of odd
non-congruent numbers $\equiv1\pmod 8$ whose prime factors are $\equiv1\pmod4$ such
that the congruent elliptic curves have second lowest Selmer
groups, which includes Li and Tian's result \cite{LT} as special cases.
\end{abstract}

\section{Introduction}
The congruent number problem is about when a positive integer can be
the area of a rational right triangle. A positive integer $n$ is a
non-congruent number is equivalent to that the congruent elliptic
curve
  \begin{equation} E:=E^{(n)}: y^2=x^3-n^2x \end{equation}
has Mordell-Weil rank zero. In \cite{Fe1} and \cite{Fe2}, Feng
obtained several series of non-congruent numbers for  $E^{(n)}$ with
the  lowest Selmer groups. In \cite{LT}, Li and Tian obtained a
series of non-congruent numbers whose prime factors are $\equiv1
\pmod 8$ such that  $E^{(n)}$ has second lowest Selmer groups. The
essential tool of the above results is the $2$-descend method of
elliptic curves. In this paper, we will use this method to get a
series of odd non-congruent numbers whose prime factors are $\equiv1
\pmod 4$ such that  $E^{(n)}$ has second lowest Selmer groups, which
includes Li and Tian's result as special cases.

Suppose $n$ is a square-free integer such that $n=p_1\cdots
p_k\equiv1\pmod 8$ and primes $p_i\equiv 1\pmod 4$, then by
quadratic reciprocity law
$\left(\frac{p_i}{p_j}\right)=\left(\frac{p_j}{p_i}\right)$.

\begin{defn}
Suppose $n=p_1\cdots p_k\equiv1\pmod 8$ and $p_i\equiv 1\pmod 4$.
The graph $G(n):=(V,A)$ associated to $n$ is a simple undirected
graph  with vertex set $V:=\{\textrm{prime}\ p\mid n\}$ and edge set
$A:=\{\overline{p q}: \left(\frac{p}{q}\right)=-1\}$.
\end{defn}

Recall for a simple undirected graph $G=(V,A)$, a partition
$V=V_0\cup V_1$ is called \emph{even} if
for any $v\in V_i$ ($ i=0,1$), $\#\{v\ra V_{1-i}\}$ is even. $G$ is called
an \emph{odd graph} if the only  even partition is the trivial
partition $V=\emptyset\cup V$. Then our main result is:
\begin{theo}\label{maintheo}
Suppose $n=p_1\cdots p_k\equiv1\pmod 8$ and $p_i\equiv 1\pmod 4$. If
the graph $G(n)$ is odd  and
$\delta(n)$ (as given by \eqref{eq:dn}) is $1$, then for the congruent
elliptic curve $E=E^{(n)}$,
 \[ \rank_\Z(E(\Q))=0\ \text{and}\  \Sha(E/\Q)[2^\infty]\cong (\Z/2\Z)^2. \]
As a consequence, $n$ is a non-congruent number.
\end{theo}

The following Corollary is Li and Tian's result, cf. \cite{LT}:
\begin{cor} \label{coro:lt}
Suppose $n=p_1\cdots p_k$ and $p_i\equiv 1\pmod 8$. If the graph
$G(n)$ is odd and the Jacobi symbol $\left(\frac{1+\sqrt{-1}}{n}\right)=-1$, then for
$E=E^{(n)}$,
 \[ \rank_\Z(E(\Q))=0\ \text{and}\  \Sha(E/\Q)[2^\infty]\cong (\Z/2\Z)^2. \]
As a consequence, $n$ is a non-congruent number.
\end{cor}

\subsection*{Acknowledgement.} This  paper  was prepared when the authors were visiting the Academy of Mathematics and Systems Science and the Morningside Center of Mathematics of Chinese Academy of Sciences, and was grew out of a project proposed by Professor Ye Tian to the second author. We would like to thank Professor Ye Tian for his vision, insistence  and  generous hospitality. We also would like to thank Jie Shu and Jinbang Yang for many helpful discussions.

\section{Review of $2$-descent method.} \label{sec:descent}
In this section, we recall the $2$-descent method of computing the
Selmer groups of elliptic curves. This section follows \cite{LT} pp 232-233,
also cf. \cite{BSD} \S 5 and \cite{Si1} X.4.

For an isogeny $\varphi: E\ra E'$ of elliptic curves defined over a number field $K$,  one has the following fundamental exact sequence
  \begin{equation} \label{eq:fun} 0\ra E'(K)/\varphi E(K)\ra S^{(\varphi)}(E/K)\ra \Sha(E/K)[\varphi]\ra 0. \end{equation}
Moreover, if $\psi: E'\ra E$ is another isogeny, for the composition $\psi\circ\varphi: E\ra E$, then
the following diagram of exact sequences commutes (cf. \cite{XZ} p 5):
  \begin{equation*}\label{eqn1}\xymatrix{
           &0\ar@{.>}^{\iota_1}[d]                                    &0\ar@{.>}^{\iota_2}[d]                             &0                           \ar[d]& \\
    0\ar[r]&E'(K)/\varphi E(K)   \ar[r]\ar[d]^{\psi} & S^{(\varphi)}(E/K)    \ar[r]\ar[d]&\Sha(E/K)[\varphi]    \ar[r]\ar[d]&0\\
    0\ar[r]&E(K)/\psi\varphi E(K)\ar[r]\ar[d]        & S^{(\psi\varphi)}(E/K)\ar[r]\ar[d]&\Sha(E/K)[\psi\varphi]\ar[r]\ar[d]&0\\
    0\ar[r]&E(K)/\psi E'(K)      \ar[r]\ar[d]        & S^{(\psi)}(E'/K)       \ar[r]      &\Sha(E'/K)[\psi]      \ar[r]      &0\\
           &0                                        &                                 &                                  &
  }\end{equation*}
Now suppose $n$ is a fixed odd positive square-free
integer, $K=\Q$, and $E/\Q$, $E'/\Q$, $\varphi$, $\psi=\varphi^{\vee}$ are given by
  \[E=E^{(n)}: y^2=x^3-n^2x,\quad E'=\widehat{E^{(n)}}: y^2=x^3+4n^2x,\]
  \[\varphi: E\ra E',\ (x,y)\mapsto(\frac{y^2}{x^2},\frac{y(x^2+n^2)}{x^2}),\]
  \[\psi: E'\ra E,\ (x,y)\mapsto(\frac{y^2}{4x^2},\frac{y(x^2-4n^2)}{8x^2}).\]
Then $\varphi\psi=[2], \psi\varphi=[2]$.  In this case $\iota_1$ and $\iota_2$ are exact.
Let $\tilde S^{(\psi)}(E'/\Q)$ denote the image of
$S^{(\psi\varphi)}(E/\Q)$ in $S^{(\psi)}(E'/\Q)$. Then
  \[\# \Sha(E/\Q)[\varphi]=\frac{\#S^{(\varphi)}(E/\Q)}{\#E'(\Q)/\varphi E(\Q)},\quad
    \# \Sha(E'/\Q)[\psi]=\frac{\#S^{(\psi)}(E'/\Q)}{\#E(\Q)/\psi E'(\Q)},\]
and
 \begin{equation} \# \Sha(E/\Q)[2]=\frac{\#S^{(\varphi)}(E/\Q)\cdot \#\tilde S^{(\psi)}(E'/\Q)}{\#E'(\Q)/\varphi E(\Q)\cdot \#E(\Q)/\psi E'(\Q)}. \end{equation}
Similarly, 
 \begin{equation} \label{eq:4} \# \Sha(E'/\Q)[2]=\frac{\#S^{(\psi)}(E'/\Q)\cdot \#\tilde S^{(\varphi)}(E/\Q)}{\#E(\Q)/\psi E'(\Q)\cdot \#E'(\Q)/\varphi E(\Q)}.\end{equation}

The $2$-descent method to compute the Selmer groups
$S^{(\varphi)}(E/\Q)$ and $S^{(\psi)}(E'/\Q)$ is as follows (cf. \cite{Si1} for general elliptic curves). Let
  \[S=\{\textrm{prime factors of}\ 2n\}\cup\{\infty\}, \]
  \[\Q(S,2)=\{b\in\Q^\times/{\Q^{\times 2}}: 2\mid \ord_p(b),\forall p\not\in S\}.\]
Note that $\Q(S,2)$ is represented by factors of $2n$ and we identify these two sets. By the exact sequence
  \[ 0\ra E'(\Q)/\varphi E(\Q)\sto[i] \Q(S,2)\sto[j]  WC(E/\Q)[\varphi],\]
where
  \[ \begin{split} & i:\  (x,y)\mapsto x, \  O\mapsto 1,\  (0,0)\mapsto 4n^2,\qquad
   j:   d\mapsto\{C_d/\Q\}\end{split}\]
and $C_d/\Q$ is the homogeneous space for $E/\Q$ defined by the equation
  \begin{equation}\label{eq:Cd} C_d: dw^2=d^2+4n^2z^4, \end{equation}
the $\varphi$-Selmer group $S^{(\varphi)}(E/\Q)$ is then
  \begin{equation} S^{(\varphi)}(E/\Q)\cong\{d\in \Q(S,2) : C_d(\Q_p)\neq\emptyset,\ \forall p\in S\}. \end{equation}
Similarly, suppose
    \begin{equation}\label{eq:Cd1} C'_d: dw^2=d^2- n^2z^4.\end{equation}
The $\psi$-Selmer group $S^{(\psi)}(E'/\Q)$ is then
  \begin{equation} S^{(\psi)}(E'/\Q)\cong\{d\in \Q(S,2) : C'_d(\Q_p)\neq\emptyset,\ \forall p\in S\}. \end{equation}

The method to compute $\tilde S^{(\varphi)}(E/\Q)$ follows from \cite{BSD} \S 5, Lemma 10:
\begin{lemma} \label{lemma:descent}
Let $d\in S^{(\varphi)}(E/\Q)$. Suppose  $(\sigma, \tau,\mu)$ is a nonzero
integer solution of  $d\sigma^2=d^2
\tau^2+4n^2\mu^2$. Let $\cM_b$ be the curve corresponding to
$b\in\Q^\times/{\Q^\times}^2$ given by
  \begin{equation} \label{eq:mb} \cM_b:\ dw^2=d^2t^4+4n^2z^4,\ \ d\sigma w-d^2\tau t^2-4n^2\mu z^2=bu^2.\end{equation}
Then $d\in\tilde S^{(\varphi)}(E/\Q)$ if and only if there exists
$b\in\Q(S,2)$ such that $\cM_b$ is locally solvable everywhere.
\end{lemma}

Note that the existence of $\sigma, \tau,\mu$ follows from
Hasse-Minkowski theorem (cf. \cite{Se}).

\section{Local computation}
We need a modification of the Legendre symbol. For $x\in \Q_p$ or
$\in \Q$ such that $\ord_p(x)$ is even, we set
 \begin{equation} \label{eq:legendre} \left (\frac{x}{p}\right):=\left (\frac{xp^{-\ord_p(x)}}{p}\right). \end{equation}
Thus $(\frac{\ }{p})$ defines a homomorphism from $\{x\in
\Q^\times/\Q^{\times 2} : \ord_p(x)\ \textrm{is even}\}$ to $\{\pm
1\}$.

\subsection{Computation of Selmer groups}
In this  subsection, we will find the conditions when $C_d$ or
$C_d'$ is locally solvable. We will not give details since one only
need to consider the valuations and quadratic residue.

\begin{lemma}\label{lem:oddphi}
$d\in S^{(\varphi)}(E/\Q)$ if and only if $d$ satisfies
\begin{enumerate}
 \item $d>0$ has no prime factor $p \equiv3\pmod 4$;
\item $\left(\frac{n/d}{p}\right)=1$ for all odd $p\mid d$; \item $\left(\frac{d}{p}\right)=1$ for all odd $p\mid (2n/d)$; \item if $2\mid d$,  $n\equiv\pm1\pmod8$.
\end{enumerate}
\end{lemma}
\begin{proof} In this case $C_d: dw^2=d^2t^4+4n^2z^4$.
It is obvious that $C_d(\R)\neq\emptyset\Leftrightarrow d>0$. Assume $d>0$.

(i) If $2\nmid d\mid n$, then $C_d: w^2=d(t^4+4(n/d)^2z^4)$.
\begin{itemize}
 \item $p=2$. $C_d(\Q_2)\neq\emptyset\Lra d\equiv1\pmod 4$.
 \item $p\mid d$. $C_d(\Q_p)\neq\emptyset\Lra \left(\frac{n/d}{p}\right)=1$ and $p\equiv1\pmod 4$.
 \item $p\nmid d$. $C_d(\Q_p)\neq\emptyset\Lra \left(\frac{d}{p}\right)=1$.
\end{itemize}

(ii) If $2\mid d\mid 2n$, then $C_d: w^2=d(t^4+(2n/d)^2z^4)$.
\begin{itemize}
 \item $p=2$. $C_d(\Q_2)\neq\emptyset\Lra d\equiv2\pmod 8,\ n\equiv\pm1\pmod 8$.
 \item $2\neq p\mid d$. $C_d(\Q_p)\neq\emptyset\Lra \left(\frac{n/d}{p}\right)=1$ and $p\equiv1\pmod 4$.
 \item $p\nmid d$. $C_d(\Q_p)\neq\emptyset\Lra \left(\frac{d}{p}\right)=1$.
\end{itemize}
Combining (i) and (ii) follows the lemma.
\end{proof}

\begin{lemma}\label{lem:oddpsi}
$d\in S^{(\psi)}(E'/\Q)$ if and only if $d$ satisfies
\begin{enumerate} \item $d\equiv\pm1\pmod 8$ or $n/d \equiv \pm1 \pmod 8$
\item $\left(\frac{n/d}{p}\right)=1$  for all $p\mid d, p\equiv 1\pmod 4$;
 \item $\left(\frac{d}{p}\right)=1$ for all $p\mid (n/d), p\equiv 1\pmod 4$.
\end{enumerate}
\end{lemma}
\begin{proof} In the case $C'_d: dw^2=d^2t^4-n^2z^4$.

(i) If $2\mid d$, consider the $2$-valuation of each side, we see $C'_d(\Q_2)=\emptyset$.

(ii) If $2\nmid d\mid n$, then $C'_d: w^2=d(t^4-(n/d)^2z^4)$.
\begin{itemize}
 \item $p=2$. $C'_d(\Q_2)\neq\emptyset \Lra d\equiv\pm1\pmod 8$ or $n/d\equiv\pm1\pmod 8$.
 \item $p\mid d$. $C'_d(\Q_p)\neq\emptyset\Lra \left(\frac{n/d}{p}\right)=1$ or $\left(\frac{-n/d}{p}\right)=1$.
 \item $p\nmid d$. $C'_d(\Q_p)\neq\emptyset\Lra \left(\frac{d}{p}\right)=1$ or $\left(\frac{-d}{p}\right)=1$.
\end{itemize}
Combining (i) and (ii) follows the lemma.
\end{proof}

\subsection{Computation of the images of Selmer groups}
Suppose $0<2d\in S^{(\varphi)}(E/\Q)$, $d$ is odd with no $\equiv
3\pmod4$ prime factor, we want to find a necessary condition for
$2d\in \tilde S^{(\varphi)}(E/\Q)$. Write $2d=\tau^2+\mu^2$ and
select the triple $(\sigma,\tau,\mu)$ in Lemma~\ref{lemma:descent}
to be $(2n, n\tau/d, \mu)$. Then the defining equations of
$\cM_{4ndb}$  in \eqref{eq:mb} can be written as
 \begin{equation} \label{eq:mb1} w^2=2d(t^4+(n/d)^2 z^4),\quad w-\tau t^2-(n/d)\mu z^2 =b u^2.
 \end{equation}
By abuse of notations, we denote the above curve by $\cM_b$.
We use the notation $O(p^m)$ to denote a number with $p$-adic valuation $\geq m$.

\vskip 0.3cm \noindent \textbf{The case $p\mid d$.} For $i_p
\equiv\tau/\mu\pmod{p\Z_p}$, $i_p\in\Z_p$ and $i_p ^2=-1$, then
  \[p\mid (\tau-i_p \mu),\quad p\nmid (\tau+i_p \mu). \]
It's easy to see $v(t)=v(z)$, we may assume that $z=1,\ t^2\equiv
\pm \frac{i_p n}{d}\pmod p$, then $\cM_b$ is given by
  \[\cM_b:     w^2=2d(t^4+(n/d)^2),\quad
    w-\tau t^2-(n/d)\mu=bu^2. \]

(i) If $v(bu^2)=m\geq3$, then by $w^2=(\tau t^2+\frac{n\mu}{d}+O(p^m))^2=2d(t^4+\frac{n^2}{d^2})$,
  \[ \Bigl(\mu t^2-\frac{n\tau}{d}\Bigr)^2=O(p^m).\]
Let $t^2=\frac{n\tau}{d\mu}+\beta$, where
$v(\beta)=\alpha\geq\frac{m}{2}$, then
  \[\begin{split}
    w^2=&2d\left((\frac{n}{d})^2+(\frac{n\tau}{d\mu})^2+2\frac{n\tau}{d\mu}\beta+\beta^2\right)\\
    =&\frac{4n^2}{\mu^2}(1+\frac{\tau\mu}{n}\beta+\frac{d\mu^2}{2n^2}\beta^2),
  \end{split}\]
 Take the square root on both sides,  then
  \[\begin{split}
    w=&\pm\frac{2n}{\mu}\left(1+\frac{1}{2}(\frac{\tau\mu}{n}\beta+\frac{d\mu^2}{2n^2}\beta^2)-\frac{1}{8}(\frac{\tau\mu}{n}\beta)^2+O(p^{3\alpha-3})\right)\\
     =&\pm\left(\frac{2n}{\mu}+\tau\beta+n\mu(\frac{\mu\beta}{2n})^2+O(p^{3\alpha-2})\right),\end{split}\]
 but on the other hand,
 \[ w    =\tau t^2+\frac{n\mu}{d}+bu^2     =\frac{2n}{\mu}+\tau\beta+bu^2. \]
The sign must be positive and
  \[bu^2=n\mu(\frac{\mu\beta}{2n})^2+O(p^{3\alpha-2}),\]
thus $p\mid b$, $\left(\frac{b/p}{p}\right)=\left(\frac{n\mu/p}{p}\right)$,
$\left(\frac{n/b}{p}\right)=\left(\frac{\mu}{p}\right)=\left(\frac{2\tau}{p}\right)$.

(ii) If $v(bu^2)=m\leq2$ and $t^2\equiv \frac{i_p n}{d} \pmod p$, let
$t^2=\frac{i_p n}{d} +p\alpha i_p $, then
  \[    w^2=2d\cdot p\alpha i_p \cdot \Bigl( \frac{2i_p n}{d} +p\alpha i_p \Bigr)
       =-4p^2\cdot\frac{n\alpha}{p}\Bigl(1+\frac{pd\alpha}{2n}\Bigr), \]
and
 \[\begin{split} w_1=& \frac{w}{p}=\pm 2i_p \sqrt{\frac{n\alpha}{p}}\Bigl(1+\frac{pd\alpha}{4n}+O(p^2)\Bigr),\\
 bu^2    =&w-\tau t^2-\frac{n\mu}{d}\\
    =&\pm 2pi_p \sqrt{\frac{n\alpha}{p}}\Bigl(1+\frac{pd\alpha}{4n}\Bigr)-\frac{i_p \tau n}{d}-\frac{n\mu}{d}-\tau\alpha i_p p+O(p^3)\\
    =&-\frac{p^2 i_p \tau}{n}\Bigl(\sqrt{\frac{n\alpha}{p}}\mp\frac{n}{p\tau}\Bigr)^2-
        \frac{ni_p }{2d\tau}(\tau-i_p \mu)^2 \pm 2p^2i_p \sqrt{\frac{n\alpha}{p}}\frac{d\alpha}{4n}+O(p^3).
  \end{split}\]
If $v(bu^2)=2$, then $\sqrt{\frac{n\alpha}{p}}\equiv \pm
\frac{n}{p\tau}\pmod p$, and
  \[\begin{split}
    bu^2=&-\frac{ni_p }{2d\tau}(\tau-i_p \mu)^2\pm2p^2i_p \sqrt{\frac{n\alpha}{p}}\frac{d\alpha}{4n}+O(p^3)\\
        =&\frac{-ni_p (\tau-i_p \mu)^3(3\tau+i_p \mu)}{8d\tau^3}+O(p^3)\\
        =&\frac{-ni_p (\tau-i_p \mu)^3}{2d\tau^2}+O(p^3)=O(p^3),
  \end{split}\]
which is impossible! Thus $v(bu^2)=1$ and  $p\mid b$,
  \[\left(\frac{b/p}{p}\right)=\left(\frac{-pi_p \tau/n}{p}\right)=\left(\frac{2p\tau/n}{p}\right),
    \ \text{or}\  \left(\frac{n/b}{p}\right)=\left(\frac{2\tau}{p}\right).\]

(iii) If $v(bu^2)=m\leq2$ and $t^2\equiv -i_p (n/d)\pmod p$, then
  \[\begin{split}
    bu^2=&w-\tau t^2-(n/d)\mu
    =(\tau i_p -\mu)n/d+O(p)\\
    =&2i_p \tau n/d+O(p)
    =(1+i_p )^2\cdot\frac{n}{d}\cdot \tau+O(p),
  \end{split}\]
thus $p\nmid b$ and
$\left(\frac{b}{p}\right)=\left(\frac{\tau}{p}\right)\left(\frac{n/d}{p}\right).$

Note that $2\tau\equiv\tau+\mu i_p\pmod p$ and
$\left(\frac{2n/d}{p}\right)=1$, hence we have
\begin{lemma}\label{lem:image1}
The curve $\cM_b$ defined by \eqref{eq:mb1} is locally solvable at $p\mid d$ if and only if
  \[ \textrm{either}\ \ p\mid  b,\ \left(\frac{n/b}{p}\right)=\left(\frac{\tau+\mu i_p}{p}\right); \quad
  \textrm{or}\ \  p\nmid b, \ \left(\frac{b}{p}\right)=\left(\frac{\tau+\mu i_p}{p}\right).\]
\end{lemma}

\vskip 0.3cm
\noindent \textbf{The case $p\mid \frac{n}{d}$.} In this case $t$ is a $p$-adic unit if and only if  $w$ is so.

(i) If $v(w)=v(t)=0$, then $w\equiv \pm\sqrt{2d}t^2\pmod p$ and
$(\pm\sqrt{2d}-\tau)t^2\equiv bu^2\pmod p$. Since
$(\sqrt{2d}-\tau)(\sqrt{2d}+\tau)=2d-\tau^2=\mu^2$ and $\sqrt{2d}\pm
\tau$ are co-prime, $\ord_p(\sqrt{2d}-\tau)$ is even and
$\left(\frac{\sqrt{2d}-\tau}{p}\right)$ is well defined. Then
$\cM_b$ is locally solvable if and only if
  \[p\nmid b, \left(\frac{2d}{p}\right)=1\ \textrm{and}\ \left(\frac{b}{p}\right)=\left(\frac{\sqrt{2d}-\tau}{p}\right).\]

(ii) If $v(z)=0$ and $w=pw_1, t=pt_1$, then
$w_1^2=2d(p^2t_1^2+(\frac{n}{pb})^2z^4)$, $w_1\equiv
\pm\sqrt{2d}\frac{n}{pd}z^2\pmod p$ and $bu^2/p\equiv
(\pm\sqrt{2d}-\mu)\frac{n}{pd}z^2\pmod p$. Thus $\cM_b$ is locally
solvable if and only if
  \[p\mid b, \left(\frac{2d}{p}\right)=1\ \textrm{and}
    \ \left(\frac{n/(db)}{p}\right)=\left(\frac{\sqrt{2d}-\mu}{p}\right).\]

Note that
 \[ 2(\sqrt{2d}-\tau)(\sqrt{2d}-\mu)=(\tau+\mu-\sqrt{2d})^2\Rightarrow
 \left(\frac{\sqrt{2d}-\mu}{p}\right)=\left(\frac{2(\sqrt{2d}-\tau)}{p}\right). \]

From now on, suppose $n=p_1\cdots p_k\equiv1\pmod 8$ and $p_i\equiv
1\pmod 4$. Pick $i_p\in \Z_p$ such that $i_p^2=-1$, then
  \[ \sqrt{2d}-\tau
  =-(\tau+\mu i_p)\cdot\frac{1}{2}\Bigl(1-\frac{\sqrt{2d}}{\tau+\mu i_p}\Bigr)^2. \]
Note that $\left(\frac{2d}{p}\right)=1$, we have
\begin{lemma}\label{lem:image2}
$\cM_b$ defined by \eqref{eq:mb1} is locally solvable at $p\mid \frac{n}{d}$ if and only if
  \[\begin{split}
    p\mid  b,&\quad \left(\frac{2d}{p}\right)=1\ \textrm{and}\ \left(\frac{n/b}{p}\right)=\left(\frac{\tau+\mu i_p}{p}\right)\left(\frac{2}{p}\right),\\
    \textrm{or}\  p\nmid b,&\quad \left(\frac{2d}{p}\right)=1\ \textrm{and}\ \left(\frac{b}{p}\right)=\left(\frac{\tau+\mu i_p}{p}\right)\left(\frac{2}{p}\right).\\
  \end{split}\]
\end{lemma}

By Lemmas~\ref{lemma:descent},~\ref{lem:oddphi},~\ref{lem:image1}
and \ref{lem:image2}, and  we have
\begin{pro}\label{prop:dn} Suppose $n=p_1\cdots p_k\equiv1\pmod 8$ and $p_i\equiv
1\pmod 4$, then $2d\in S^{(\varphi)}(E/\Q)$ if and only if $d>0$ and
$\left(\frac{2n/d}{p}\right)=1$ for $p\mid d$,
$\left(\frac{2d}{p}\right)=1$ for $p\mid \frac{n}{d}$. In this case $2d\in
\tilde{S}^{(\varphi)}(E/\Q)$ only if there exists $b\in \Q(S,2)$
satisfying:

(1) If $p\mid d$, $i_p\equiv \tau/\mu\pmod{p\Z_p}$, $i_p^2=-1$,
  \[  p\mid  b,\ \left(\frac{n/b}{p}\right)=\left(\frac{\tau+\mu i_p}{p}\right),\quad
  \textrm{or}\quad  p\nmid b,\quad \left(\frac{b}{p}\right)=\left(\frac{\tau+\mu i_p}{p}\right).\]

(2) If $p\mid \frac{n}{d}$,   $i_p^2=-1$,
  \[  p\mid  b,\ \left(\frac{n/b}{p}\right)=\left(\frac{2(\tau+\mu i_p)}{p}\right),\quad
  \textrm{or}\quad  p\nmid b,\quad \left(\frac{b}{p}\right)=\left(\frac{2(\tau+\mu i_p)}{p}\right).\]
\end{pro}

\section{Proof of the main result}
\subsection{Some facts about graph theory.}
We now recall some notations and results in graph theory, cf. \cite{Fe1, Fe2}.

\begin{defn}\label{matofgraph}
Let $G=(V,A)$ be a simple undirected graph. Suppose $\# V=k$.  The
\emph{adjacency matrix} $M(G)=(a_{ij})$ of $G$ is the $k\times k$
matrix defined as
  \begin{equation} a_{ij}:=\begin{cases}
    0, &\ \textrm{if}\ \overline{v_i v_j}\not\in A;\\
    1, &\ \textrm{if}\ \overline{v_i v_j}\in A.
  \end{cases}\end{equation}
The \emph{Laplace matrix} $L(G)$ of $G$ is defined as
  \begin{equation} L(G)=\diag\{d_1,\ldots,d_k\}-M(G) \end{equation}
where $d_i$ is the degree of $v_i$.
\end{defn}

\begin{theo}\label{graphthm}
Let $G$ be a simple undirected graph and $L(G)$ its Laplace matrix. \\
\indent (1) The number of even partitions of $V$ is $2^{k-1-r}$, where $r=\rank_{\F_2} L(G)$.\\
\indent (2) The graph $G$ is odd if and only if  $r=k-1$.\\
\indent (3) If $G$ is odd, then the equations
  \[L(G)\left(\begin{smallmatrix}c_1\\ \vdots \\c_k\end{smallmatrix}\right)
  =\left(\begin{smallmatrix}t_1\\ \vdots \\t_k\end{smallmatrix}\right)\]
has solutions if and only if $t_1+\cdots+t_k=0$.
\end{theo}
\begin{proof}
The proof of the first two parts follows from \cite{Fe1}.
We have a bijection
  \[\begin{split} \F_2^k/\{(0,\cdots,0),(1,\cdots,1)\} &\stackrel{\sim}{\longrightarrow} \{\textrm{partitions of $V$}\}\\
  (c_1,\ldots,c_k)&\longmapsto (V_0, V_1)\end{split}\]
where $V_i=\{v_j: c_j=i\ (1\leq j\leq k)\},\ i\in\{0,1\}$.

Regard $L(G)=\diag\{d_1,\ldots,d_k\}-(a_{ij})$ as a matrix over $\F_2$. If
  \[L(G)\left(\begin{smallmatrix}c_1\\ \vdots \\c_k\end{smallmatrix}\right)=\left(\begin{smallmatrix}b_1\\ \vdots \\b_k\end{smallmatrix}\right)\in\F_2^k,\]
then if $v_i\in V_t, t\in\{0,1\}$,
  \[\begin{split}
    b_i&=d_ic_i+\sum_{j=1}^k a_{ij}c_j=\sum_{j=1}^k a_{ij}(c_i+c_j)\\
    &=\sum\limits_{j=1}^k a_{ij}(t+c_j)=\sum\limits_{c_j=1-t}a_{ij}=\#\{v_i\ra V_{1-t}\}\in\F_2.
  \end{split}\]

(1) The number of even partitions is
  \[\frac{1}{2}\#\left\{(c_1,\ldots,c_k)\in\F_2^n : L(G)\left(\begin{smallmatrix}c_1\\ \vdots \\c_k\end{smallmatrix}\right)
  =\left(\begin{smallmatrix}0\\ \vdots \\0\end{smallmatrix}\right)\right\}=2^{k-1-r}.\]

(2) follows from (1) easily.

(3) Since $L$ is of rank $k-1$, the image space of $L$ is of
dimensional $k-1$, but it lies in the hyperplane $x_1+\cdots+x_k=0$,
thus they coincide and the result follows.
\end{proof}

\subsection{Graph $G(n)$ and Selmer groups of $E$ and $E'$.} From now on, we suppose
 \begin{quote} $n=p_1\cdots p_k\equiv1\pmod 8$ and $p_i\equiv 1\pmod 4$. \end{quote}
Recall for an integer $a$ prime to $n$, the Jacobi symbol
$\left(\frac{a}{n}\right)=\prod_{p\mid n} \left(\frac{a}{p}\right)$, which is extended
to a multiplicative homomorphism from $\{a\in
\Q^{\times}/\Q^{\times 2} : \ord_p(a)\ \text{even for } p\mid n\}$
to $\{\pm 1\}$. Set
 \begin{equation} \left[\frac{a}{n}\right]:=\frac{1}{2}\Bigl(1-\left(\frac{a}{n}\right)\Bigr). \end{equation}
The symbol $[\frac{}{n}]$ is an additive homomorphism from
$\{a\in \Q^{\times}/\Q^{\times 2} : \ord_p(a)\ \text{even} \    \text{for} \
p\mid n\}$ to $\F_2$.

By definition, the adjacency matrix $M(G(n))$ has entries $a_{ij}= \left[\frac{p_i}{p_j}\right]$. For $0<d\mid n$, we denote by $\{d,\frac{n}{d}\}$ the partition
$\{p:p\mid d\}\cup\{p:p\mid \frac{n}{d}\}$ of $G(n)$.

The following proposition is a translation of results in Lemma~\ref{lem:oddphi} and Lemma~\ref{lem:oddpsi}:
\begin{pro}\label{lem:6.3}  Given a factor $d$ of $n$.

(1) For the Selmer group $S^{(\varphi)}(E/\Q)$,
\begin{itemize}
 \item[(1-a)] $d\in S^{(\varphi)}(E/\Q)$ if and only if $d>0$ and $\{d, n/d\}$ is an even partition of $G(n)$;
 \item[(1-b)] Suppose
  \[ c_i=\begin{cases} 1,\ &\text{if}\ p_i\mid d,\\ 0,\ &\text{if}\ p_i\mid \frac{n}{d}; \end{cases} \qquad
  t_i=\left[\frac{2}{p_i}\right]. \]
 Then $2d\in S^{(\varphi)}(E/\Q)$ if and only if $d>0$ and
  \[L(G)\left(\begin{smallmatrix}c_1\\ \vdots \\c_k\end{smallmatrix}\right)
  =\left(\begin{smallmatrix}t_1\\ \vdots \\t_k\end{smallmatrix}\right). \]
\end{itemize}

(2) For the Selmer group  $S^{(\psi)}(E'/\Q)$,
\begin{itemize}
 \item[(2-a)]  $d\in S^{(\psi)}(E'/\Q)$ if and only if $d\equiv\pm1\pmod 8$ and $\{d,n/d\}$ is an even partition of $G(n)$;
 \item[(2-b)] $2d\notin S^{(\psi)}(E'/\Q)$.
 \end{itemize}
\end{pro}
\begin{proof} One only has to show (1-b), the rest is easy. For any $i$, let $[i]$  be the set of $j$ such that $p_i$ and $p_j$ are both prime divisors of $d$ or $n/d$. Then
 \[ d_i c_i+\sum_{j\neq i} a_{ij} c_j=\sum_{j\neq i} a_{ij}(c_i+c_j)=\sum_{j\notin [i]} a_{ij}=\left[\frac{d}{p_i}\right]\ \text{or}\ \left[\frac{n/d}{p_i}\right]. \]
Then (1-b) follows from Lemma~\ref{lem:oddphi}.
\end{proof}

Applying  Theorem~\ref{graphthm}(3) to Proposition~\ref{lem:6.3}, then we have
\begin{cor} \label{cor:d}
If $G(n)$ is odd, there exists a unique factor $0<d<\sqrt{2n}$ of $n$ such that
  \[ S^{(\varphi)}(E/\Q)=\{1,2d,2n/d,n\}\cong \Z/2\Z\times \Z/2\Z,\]
and
  \[ S^{(\psi)}(E'/\Q)=\{\pm1, \pm n\}\cong \Z/2\Z\times \Z/2\Z. \]
\end{cor}
For the $d$ given in Corollary~\ref{cor:d}, write $2d=\tau^2+\mu^2$. If $2d\in \tilde{S}^{(\varphi)}(E/\Q)$, we suppose $b$ satisfies the condition that $\cM_b$ defined by \eqref{eq:mb1} is locally solvable everywhere.
Suppose $c'=(c'_1,\cdots, c'_k)^T$ and $t'=(t'_1,\cdots t'_k)^T$ are given by
   \[ c'_j=\begin{cases} 1,\ &\text{if}\ p_j\mid b,\\ 0,\ &\text{if}\ p_j\nmid b; \end{cases} \qquad t'_j=\begin{cases} \Big[\dfrac{\tau+\mu i_{p_j}}{p_j}\Big],\ &\text{if}\ p_j\mid d,\\ \Big[\dfrac{2(\tau+\mu i_{p_j})}{p_j}\Big],\ &\text{if}\ p_j\mid \frac{n}{d}. \end{cases}   \]
By Proposition~\ref{prop:dn}, $L c'=t'$, i.e., $Lv=t'$ has a solution $v=c'$, which means that the summation of $t'_j$ must be zero in $\F_2$ by  Theorem~\ref{graphthm}(3).

\begin{defn} \label{defn:dn} Suppose $n$ is given such that $G(n)$ is an odd graph. For the unique factor  $d$ given in Corollary~\ref{cor:d}, write $2d=\tau^2+\mu^2$ and  $\frac{2n}{d}=\tau'^2+\mu'^2$, Let $i\in\Z/n\Z$ be defined by
  \begin{equation}i\equiv\frac{\tau}{\mu}\pmod d,\quad i\equiv\frac{\tau'}{\mu'}\pmod{\frac{n}{d}}.\end{equation}
We define
  \begin{equation} \label{eq:dn} \delta(n):=\left[\frac{\tau+\mu i}{n}\right]+  \left[\frac{2}{d}\right]
  \in\F_2.\end{equation}
\end{defn}

Then the following is a consequence of Proposition~\ref{prop:dn}.

\begin{cor}\label{cor:image}
If $G(n)$ is odd and $\delta(n)=1$, then
  \[\tilde S^{(\varphi)}(E/\Q)=\{1\}.\]
\end{cor}
\begin{proof}
Let $\lambda^*$ be the $\F_2$-rank of $\tilde S^{(\varphi)}(E/\Q)$, $\lambda$ be the $\F_2$-rank of $S^{(\varphi)}(E/\Q)$, then $\lambda=2$. The existence of the Cassels' skew-symmetric bilinear form on $\Sha$ implies that the difference $\lambda-\lambda^*$ is even.

By the above analysis, $\delta(n)=\sum\limits_j t'_j\neq 0$, thus $2d\notin\tilde S^{(\varphi)}(E/\Q)$,
we have $\lambda^*<\lambda$, $\lambda^*=0$.
\end{proof}

\begin{remark}
If we replace $d$ by $\frac{n}{d}$ in the definition, $\delta(n)$ is invariant. Indeed, $[\frac{2}{d}]=[\frac{2}{n/d}]$. For the other term,
  \[\left[\frac{\tau+\mu i}{n}\right]=\left[\frac{\tau+\mu i}{d}\right]+\left[\frac{\tau+\mu i'}{n/d}\right]\]
where $i\equiv \tau/\mu\pmod d$, $i'\equiv\tau'/\mu'\pmod{n/d}$. Let
$u=(\tau\tau'-\mu\mu')/2$, $v=(\tau\mu'-\mu\tau')/2$, then
  \[\begin{split}
    u+vi=&(\tau+\mu i)(\tau'+\mu' i)/2 \equiv \tau(\tau'+\mu'\cdot\frac{\tau}{\mu}) \\ \equiv& \tau\mu(\tau'\mu+\tau\mu')/\mu^2 \equiv(\tau+\mu)^2/\mu^2\cdot v/2\pmod d.
  \end{split}\]
Similarly, $u+vi'\equiv (\tau'+\mu')^2/\mu'^2\cdot v/2\pmod{(n/d)}$.
If we interchange $d$ and $n/d$, $\delta(n)$ will differ
  \[\begin{split}&\left[\frac{\tau+\mu i}{d}\right]+\left[\frac{\tau+\mu
    i'}{n/d}\right]+\left[\frac{\tau'+\mu' i'}{n/d}\right]+\left[\frac{\tau'+\mu' i}{d}\right]\\
    =&\left[\frac{2(u+vi)}{d}\right]+\left[\frac{2(u+vi')}{n/d}\right]=\bigg[\frac{v}{d}\bigg]
    +\left[\frac{v}{n/d}\right]\\
    =&\left[\frac{v}{n}\right]=\left[\frac{n}{v}\right]=0\in\F_2.
    \end{split}\]
Thus $\delta(n)$ does not change, which implies that $\delta(n)$
does not depend on the choice of $d,\tau,\mu$ and only depend on
$n$.
\end{remark}
\subsection{Proof of the main result.}
\begin{proof}[Proof of Theorem~\ref{maintheo}]
We shall use the fundamental exact sequence \eqref{eq:fun} and the commutative diagram in \S\ref{sec:descent} frequently.

Since $E(\Q)_{\tor}\cap\psi E'(\Q)=\{O\}$ and $\# E(\Q)_{\tor}=4$, $\# E(\Q)/\psi E'(\Q)\geq4$. Since $G(n)$ is odd, $\# S^{(\psi)}(E'/\Q)=4$ and  $\# E(\Q)/\psi E'(\Q)=4$, by \eqref{eq:fun}, $\Sha(E'/\Q)[\psi]=0$. Apparently $\tilde S^{(\psi)}(E'/\Q)\supseteq E(\Q)/\psi E'(\Q)$ and thus $\#\tilde S^{(\psi)}(E'/\Q)=4$.

By Corollary~\ref{cor:image}, $\tilde{S}^{(\varphi)}(E/\Q)=\{1\}$, then
$\# E'(\Q)/\varphi E(\Q)=1$. The facts $\# E(\Q)/\psi E'(\Q)=4$ and
$E(\Q)_{\tor}\cong (\Z/2\Z)^2$ imply that $\#E(\Q)/2 E(\Q)=4$ and
  \[ \rank_\Z E(\Q)=\rank_\Z E'(\Q)=0. \]
From $\Sha(E'/\Q)[\psi]=E'(\Q)/\varphi E(\Q)=0$, the diagram tells us that
  \[ \Sha(E/\Q)[2]\cong \Sha(E/\Q)[\varphi]\cong S^{(\varphi)}(E/\Q) \cong\Z/2\Z\times\Z/2\Z, \]
and \eqref{eq:4} tells us that
  \[ \Sha(E'/\Q)[2]\cong\Sha(E'/\Q)[\psi]\cong 0.  \]
Hence $\Sha(E'/\Q)[2^\infty]=0$ and $
\Sha(E'/\Q)[2^k \psi]=0$. By the exact sequence
  \[0\ra\Sha(E/\Q)[\varphi]\ra\Sha(E/\Q)[2^k]\ra\Sha(E'/\Q)[2^{k-1} \psi],\]
we have for every $k\in \N_{+}$,
  \[\Sha(E/\Q)[2^k]\cong\Sha(E/\Q)[\varphi]\cong(\Z/2\Z)^2,\]
and thus $\Sha(E/\Q)[2^\infty]\cong(\Z/2\Z)^2$.
\end{proof}
\begin{proof}[Proof of Corollary~\ref{coro:lt}]
In this case, $d=1$ and $\tau=\mu=1$,
$\delta(n)=\Big[\frac{1+\sqrt{-1}}{n}\Big]$, thus the result
follows.
\end{proof}

\end{document}